\documentclass{amsart}
\usepackage{amssymb}
\usepackage{amsmath, amsfonts}
\usepackage{setspace, csquotes}
\usepackage[super]{nth}
\usepackage{xcolor}
\usepackage{mathtools}
\usepackage[colorinlistoftodos]{todonotes}

\newtheorem{theorem}{Theorem}[section]
\newtheorem{proposition}[theorem]{Proposition}
\newtheorem{lemma}[theorem]{Lemma}

\newtheorem{cor}[theorem]{Corollary}
\theoremstyle{definition}

\numberwithin{equation}{section}

\begin{document}

\title[Weighted exponential sums and its applications]
{Weighted exponential sums and its applications}
\author{Nilanjan Bag}
 \address{Department of Mathematics, Thapar Institute of Engineering and Technology, Patiala, Punjab 147004, India}
\email{nilanjanb2011@gmail.com}
\author{Dwaipayan Mazumder}
 \address{Chennai Mathematical Institute, Kelambakkam, Siruseri, Tamil Nadu 603103, India}
\email{dwmaz.1993@gmail.com}
\date{August 8, 2024}
\begin{abstract}
Let $f$ be a real polynomial with irrational leading co-efficient. In this article, we derive distribution of $f(n)$ modulo one for all $n$ with at least three divisors and also we study distribution of $f(n)$ for all square-free $n$ with at least two prime factors.
We study exponential sums when weighted by divisor functions and exponential sums over square free numbers. In particular, we are interested in evaluating 
\begin{align*}
\sum_{n\leq N}\tau(n)e\left(f(n)\right) ~\text{and}~\sum_{n\leq N}\mu^2(n)e\left(f(n)\right),
\end{align*}  
for some polynomial $f$, where $\tau$ is the divisor function and $\mu$ is the M\"{o}bius function. We get non-trivial estimates when the leading co-efficient  $\alpha$ of $f$ belongs to the \enquote{minor arc}.
\end{abstract}
\maketitle
\section{Introduction and statements of the results}
Exponential sums over polynomial phase, i.e. the sum of the following form
\begin{equation}\label{start}
\mathcal{S}_F(N)= \sum_{n\leq N} e(F(n)),
\end{equation}
where $F(x)=\alpha_0+\alpha_1x+\alpha_2x^2+...+\alpha_nx^k, \alpha_i \in \mathbb{R}$ and $e(x)=e^{2\pi ix}$ can be handled by well known Weyl and Vinogradov method \cite{IWA}. Before stating the goal of this paper, we would like to state the usual concept of major and minor arcs. The well known theorems of Dirichlet and Kronecker in the theory of Diophantine approximation have been studied for long. According to the Dirichlet approximation theorem, for $\alpha \in \mathbb{R}$, there are integers $a$ and $q$, such that
\begin{align}\label{Dio}
 \left|\alpha-\frac{a}{q}\right|\leq \frac{1}{qP}, 
\end{align}
with $1\leq q\leq P$ and $(a,q)=1$, where $P$ is a fixed positive number. Let $Q$ be a positive integer with $P\geq 2Q\geq2$. If $\alpha$ satisfies \eqref{Dio} with $q\leq Q$ we say $\alpha$ belongs to \enquote{major arc}
\begin{equation}\label{Majorarc}
    \mathcal{M}= \left\{\alpha \mid \bigg|\alpha-\frac{a}{q}\bigg|\leq \frac{1}{qP} \quad \text{with} \quad q\leq Q, (a,q)=1\right\}.
\end{equation}
The remaining part $\mathfrak{m}=[0,1]\setminus\mathcal{M}$ is called \enquote{minor arc}.\par
Exponential sums of the form \eqref{start} is a crucial quantity in analytic number theory for its various applications in different number theoretic problems. For example such sums arise in the applications of the Hardy-Littlewood circle method to the Waring Goldbach problems. For more details, see \cite{RCV}. When taking summation over primes, such sums occur in the problems involving the distribution of $\alpha p^k$ modulo one \cite{BH, KCW}.  Now let us consider the weighted version of \eqref{start}, namely
\begin{equation}\label{main}
\mathcal{S}_{F,a} = \sum_{n\leq N} a(n)e(F(n)),
\end{equation}
where $F(n)$ is a polynomial as before and $a(n)$ is a multiplicative function, with some restriction on both the multiplicative function and  the polynomial. There are numerous works and significant literatures towards such problems. The estimate of the exponential sum over primes in short intervals was first studied by I. M. Vinogradov \cite{IMV}. Take $F(n)=n^k$. Then for the case $k=1$, that is for the linear exponential sums, such sums have been 
tackled extensively for  its applications to the study of the Goldbach–Vinogradov theorem with three almost
equal prime variables. For $k=2$, one can find the work of Liu and Zhang \cite{LZ} for short intervals. Later such results have been improved by L\"{u} and Lao \cite{LL}, which is as good as previously derived results using Generalized Riemann Hypothesis. Since then such directions have attracted many authors.
\par In our work, we consider those multiplicative functions  $a(n)$ which can be written as $b*c(n)$ where \enquote{$*$} is the usual Dirichlet convolution of two arithmetic functions. For the polynomial we restrict ourselves to those $F(n)$ whose leading term belongs to \enquote{minor arc} $\mathfrak{m}$. To be specific, in this paper we restrict ourselves to $\alpha(n)= \tau(n) \quad\text{and} \quad {\mu}^2(n)$. Granville, Vaughn, Wooley and other authors \cite{GVW}  have investigated the sum when $\alpha(n)=\mu^2(n)$ and $F(n)=n$, restricted to minor arc. Instead in this paper we take any polynomial $F(n)$ with some restrictions. Our method of evaluation to the weighted sum \eqref{main} is based on \cite{HAR}. Throughout the rest of this paper we consider $$\gamma=4^{1-k}.$$
At this point we are ready to state two results where we achieve non-trivial bounds for exponential sums when we take weights by $\tau(n)$ and $\mu^2(n)$. To be specific, we prove 
\begin{theorem}\label{MT1}
Let $f(x)$ be a real valued polynomial in $x$ of degree $k\geq 2$. Suppose $\alpha $ be the leading co-efficient of $f$ and there are integers $a$ and $q$ such that 
\begin{align*}
\left|\alpha-\frac{a}{q}\right|\leq \frac{1}{qP},
\end{align*}
for some fixed $P$, with $1\leq q\leq P$ and $(a,q)=1$. Then for fix $\epsilon >0$ and $0<\theta <1$, we have 
\begin{align*}
\sum_{n\leq N} \tau(n)e(f(n))\ll \log N\cdot  N^{1+\epsilon} \bigg(\frac{1}{q}+\frac{q}{N^k}+\frac{1}{N^{1-\theta}}\bigg)^{\gamma}.
\end{align*}
\end{theorem}
\begin{cor}
   Let $Q=N^{1/2}$ in $\mathcal{M}$ defined in \eqref{Majorarc}. With the hypothesis of Theorem \ref{MT1}, $P=N$ and $\alpha \in \mathfrak{m}$ we have
$$
\sum_{n\leq N} \tau(n)e(f(n))\ll \log N \cdot N^{1+\epsilon-\frac{\gamma}{2}}.
$$
\end{cor}
In the next theorem, we replace the divisor function $\tau(n)$ by square of 
m\"{o}bius function $\mu^2(n)$ to derive the following:
\begin{theorem}\label{MT2}
Let $f(x)$ be a real valued polynomial in $x$ of degree $k\geq 2$. Suppose $\alpha $ be the leading co-efficient of $f$ and there are integers $a$ and $q$ such that 
\begin{align*}
\left|\alpha-\frac{a}{q}\right|\leq \frac{1}{qP},
\end{align*}
for some fixed $P$, with $1\leq q\leq P$ and $(a,q)=1$. Then for fix $\epsilon>0$ and $0<\theta<1$, we have 
\begin{align*}
\sum_{n\leq N} \mu^2(n)e(f(n))\ll N^{1+\epsilon} \bigg(\frac{1}{q}+\frac{q}{N^k}+\frac{1}{N^{1-\theta}}\bigg)^{\gamma}.
\end{align*}
\end{theorem}
\begin{cor}
   Let $Q=N^{1/2}$ in $\mathcal{M}$ defined in \eqref{Majorarc}. With the hypothesis of Theorem \ref{MT1}, $P=N$ and $\alpha \in \mathfrak{m}$ we have
$$
\sum_{n\leq N} \mu^2(n)e(f(n))\ll N^{1+\epsilon-\frac{\gamma}{2}}.
$$
\end{cor}
We have a real polynomial $f(x)$ of degree $k\geq 2$ with irrational leading co-efficient. For $\epsilon >0$, Harman \cite[Theorem 2]{HAR} proved that there are infinitely many solutions of $||f(p)||<p^{-\gamma/2+\epsilon}$ in primes where $\gamma$ is defined as before and $||x||$ is the distance of $x$ from the nearest integer. Also such results were previously studies by A. Ghosh \cite{AG} for some special cases. In this article, we refined the previous results and proved that there are infintely many solutions in $n$ with atleast three divisors. In particular, we prove the following. 
\begin{theorem}\label{MT3}
Let $f(x)$ be a real polynomial in $x$ of degree $k$, with irrational leading coefficient. Suppose $\epsilon>0$ is given. Then there are infinitely many solutions of 
\begin{align*}
||f(n)||<n^{-(\gamma/4)+\epsilon},
\end{align*} 
where $n$ has at least three divisors. 
\end{theorem}
In addition, we prove the existence of solutions in square free numbers 
\begin{theorem}\label{MT4}
Let $f(x)$ be a real polynomial in $x$ of degree $k$, with irrational leading coefficient. Suppose $\epsilon>0$ is given. Then there are infinitely many solutions of 
\begin{align*}
||f(n)||<n^{-(\gamma/4)+\epsilon},
\end{align*} 
where $n$ is square free and has at least two prime factors.
\end{theorem}
\section{notation and preliminaries} 
Throughout the article, we consider the following notations:
\begin{itemize}
\item $f=O(g)$ or $f\ll g$ for $f\leq c\cdot g$, where $c$ is a constant which may sometimes depend upon some $\epsilon.$ We denote $x\sim X$ for the dyadic interval $X\leq x<2X$. 
\item $\mathbb{P}$ denotes set of all primes.
\item $\tau(n)$ denotes the divisor function defined as 
\begin{align*}
\tau(n)=\sum_{d|n}1.
\end{align*}
\item $\mu(n)$ denotes the M\"{o}bius function defined as 
\begin{align*}
\mu(1)=1 
\end{align*}
and if $n>1$ and $n=p_1^{a_1}\cdots p_k^{a_k}$, then 
\begin{align*}
\mu(n)=\begin{cases}
(-1)^k, &~\text{if}~a_1=\cdots=a_k=1;\\
0, &~\text{otherwise}.
\end{cases}
\end{align*}
\item We set,
$$P(z):= \prod_{\substack{p\in \mathbb{P}\\ p<z}}p.$$

\item Let $a$ and $b$ be two arithmetical functions. Their dirichlet product denoted as $a*b$ is defined by 
\begin{align*}
a*b(n)=\sum_{d|n}a(d)b(n/d).
\end{align*} 
\item For any arithmetical function $a$, $a^{-1}$ denotes the inverse function such that $a*a^{-1}(n)=[\frac{1}{n}]$  for all $n$ where $[\bullet]$ denotes box function.
\item We have $1(n)=1$.
\item $\tau_3(n)=1*1*1(n)$.
\end{itemize} 
Thoroughout the article $p$ always denotes prime number.
The following result of Vaughan \cite{RCV} plays a crucial role in the proof of our result. 
\begin{lemma}
For any real valued function $f$ and natural number $N$, we have 
\begin{align*}
\sum_{p\leq N}(\log p)e(f(p))=O(N^{1/2})+S_1-S_2-S_3,
\end{align*}
where 
\begin{align*}
S_1&=\sum_{d\leq N^{1/3}}\mu(d)\sum_{l<Nd^{-1}}(\log l)e(f(dl)),\\
S_2&=\sum_{r\leq N^{2/3}}\phi_1(r)\sum_{m\leq Nr^{-1}}e(f(mr)),\\
S_3&=\sum_{N^{1/3}<m\leq N^{2/3}}\phi_2(m)\sum_{N^{1/3}<n\leq Nm^{-1}}\Lambda(n)e(f(mn)),
\end{align*}
and $\phi_1(r)\ll \log r$ and $\phi_2(m)\ll \tau(d)$.
\end{lemma}
\begin{proof}
For detailed proof, see \cite{RCV}.
\end{proof}
\begin{lemma}\cite[Lemma 10C]{WS}
Let $g(x)$ be a real valued polynomial of degree $k$ with leading coefficient $\beta$. Then for $\epsilon>0,$
\begin{align*}
\left|\sum_{n=1}^{X}e(g(n))\right|^{R}\ll X^{R-k+\epsilon}\sum_{y=1}^{k!X^{k-1}}\min\left(X,\frac{1}{||\beta y||}\right).
\end{align*}  
$R=2^{k-1}$.
\end{lemma}
The next two lemmas are due to Harman from his works on trigonometric sums over primes. 
\begin{lemma}\cite[Lemma 3]{HAR}\label{lem-h1}
Let $f$ be a polynomial which satisfies the conditions of the Theorem \ref{MT1}. Suppose $\epsilon>0$, and that $\phi(u)$ and $\psi(v)$ are real functions. Put
\begin{align*}
T&=\max |\psi(v)|,\\
F&=\left(\left(\sum_{u\leq W}\phi(u)\right)\right)^{1/2}.
\end{align*}
For positive integers $M$, $W$ and $X$, write
\begin{align*}
S=\sum_{u=1}^{W}\sum_{\substack{v=1\\ uv\leq M}}^{X}\phi(u)\psi(v)e(f(uv)).
\end{align*}
Then we have 
\begin{align*}
\left(\frac{S}{FT}\right)^{R^2}\ll (WX)^{R^2}\left(X^{-R}+(WX)^{-k+\epsilon}\sum_{z=1}^{Y}\min \left(W, \frac{1}{||\alpha z||}\right)\right),
\end{align*}
where $Y=X^kW^{k-1}(k!)^2$ and $R=2^{k-1}$.
\end{lemma}
As a direct corollary he got
\begin{cor}\cite[Corollary]{HAR}
    Let $a,q$ be as in Theorem \ref{MT1}. If $T=o(X^{\delta})$ and $F=o(X^{\delta})$ for every $\delta >0$, then
    \begin{align*}
        S\ll (XW)^{1+\epsilon}\left(X^{-R}+W^{-1}+q^{-1}+(XW)^{-k}q\right)^{\gamma}.
    \end{align*}
\end{cor}
\begin{lemma}\cite[Lemma 4]{HAR}\label{lem-h2}
Suppose we have the hypothesis of the last lemma and its corollary, but either 
\begin{align*}
\phi(x)&=1~~\text{or}~~\phi(x)=\log x, 
\end{align*}
for all $x$. Then we have 
\begin{align*}
S\ll (XW)^{1+\epsilon}X^{(k-1)/R}\left(q^{-1}+q(WX)^{-k}+W^{-1}\right)^{1/R}.
\end{align*}
\end{lemma}
\begin{lemma}\label{lem-h3}
Let $g(n)$ be a non-negative function for all $n$, $H$ a positive integer and $\alpha_n$ a sequence of real numbers for $n\in \mathcal{A}$, where $\mathcal{A}\subset \mathbb{N}$ and is a finite set. If
\begin{align*}
\sum_{h=1}^{H}\left|\sum_{n\in \mathcal{A}}g(n)e(h\alpha_n)\right|<\frac{1}{6}\sum_{n\in\mathcal{A}}g(n).
\end{align*} 
then there is a solution of $||\alpha_n||<H^{-1},$ with $ n\in\mathcal{A}.$
\end{lemma}
\begin{proof}
    The proof can be easily derived from  \cite{BH-II}. 
\end{proof}
Now we prove a few lemmas which play an integral part in the prove of our main theorem. 
\begin{lemma}\label{Vaughanlike}
Let $a, b, c$ be three arithmetical functions related by
$$
a(n)=b*c(n)
$$
Then for any $U, V >0$ real numbers,
\begin{align*}
a(n) = a(n)1_{[1,U]}(n) +&\sum_{\substack{lm=n\\ l\leq V}}b(l)c(m)\\
 -&\sum_{\substack{lm=n\\l\leq UV}}\sum_{\substack{rs=l\\r\leq V\\s\leq U}} a(s)b(r)b^{-1}(m)\\ -&\sum_{\substack{lm=n\\l>U\\m>V}}a(l)\sum_{\substack{r|m\\r\leq V}}b(r)b^{-1}(\frac{m}{r}).
\end{align*}
\end{lemma}
\begin{proof}
  We follow the proof of Vaughan's identity \cite{VI-H}. First if $n<U$ we have nothing to prove as the last sum is empty and 
$$
\sum_{\substack{lm=n\\ l\leq V}}b(l)c(m)-\sum_{\substack{lm=n\\l\leq UV}}\sum_{\substack{rs=l\\r\leq V\\s\leq U}} a(s)b(r)b^{-1}(m) =0
$$
So we start with $n>U$. As $a(n)=b*c(n)$ we have 
$$
 a(n)=\sum_{lm=n}b(l)c(m).
 $$
 For any $V\in \mathbb{R},V>0$ we split the summations into two parts
\begin{align*}
a(n)=\sum_{\substack{lm =n\\l\leq V}}b(l)c(m) + \sum_{\substack{lm =n\\l > V}}b(l)c(m) .
\end{align*} 
As $b^{-1}*a=c$, for the second summation we have
\begin{align*}
\sum_{\substack{lm=n\\l>V}}b(l)c(m)&=\sum_{\substack{lm=n\\l>V}}b(l)\sum_{r|m}b^{-1}(r)a(m/r)\notag\\
&=\sum_{s|n}a(s)\sum_{\substack{l|\frac{n}{s}\\l>V}}b(l)b^{-1}\left(\frac{n/s}{l}\right).
\end{align*}
Now as $n\neq s$ (otherwise $1>V$, which is definitely false for $V>1$) we have $\sum_{\substack{l|\frac{n}{s}}}b(l)b^{-1}\left(\frac{n/s}{l}\right)=0$, which gives 
\begin{align*}
\sum_{s|n}a(s)\sum_{\substack{l|\frac{n}{s}\\l>V}}b(l)b^{-1}\left(\frac{n/s}{l}\right)= - \sum_{s|n}a(s)\sum_{\substack{l|\frac{n}{s}\\l\leq V}}b(l)b^{-1}\left(\frac{n/s}{l}\right).
\end{align*}
Similarly, splitting further over $s|n$, we have 
\begin{align*}
-\sum_{s|n}a(s)\sum_{\substack{l|\frac{n}{s}\\l\leq V}}b(l)b^{-1}\left(\frac{n/s}{l}\right)=&-\sum_{\substack{s|n\\s\leq U}}a(s)\sum_{\substack{l|\frac{n}{s}\\l\leq V}}b(l)b^{-1}\left(\frac{n/s}{l}\right)\notag\\&-\sum_{\substack{s|n\\s> U}}a(s)\sum_{\substack{l|\frac{n}{s}\\l\leq V}}b(l)b^{-1}\left(\frac{n/s}{l}\right).
\end{align*}
Rearranging the first sum on the right hand side we get 
\begin{align*}
\sum_{\substack{n=rt\\r\leq UV}}\sum_{\substack{sl=r\\s\leq U\\l\leq V}}a(s)b(l)b^{-1}\left(\frac{k}{l}\right),
\end{align*}
and similarly for the second summation we have 
\begin{align*}
\sum_{\substack{ks=n\\s> U\\k>V}}a(s)\sum_{\substack{l|\frac{n}{s}\\l\leq V}}b(l)b^{-1}\left(\frac{k}{l}\right).
\end{align*}
Hence combining equation \eqref{eq2} and \eqref{eq3} into \eqref{eq1}, we complete the proof of the Lemma \ref{Vaughanlike}. 
\end{proof}
Using Lemma \ref{Vaughanlike} we can directly prove the following result.
\begin{proposition}\label{weighted}
Let $a, b, c$ be three arithmetical functions related by
$$
a(n)=b*c(n)
$$
Then 
$$
\sum_{n\leq N}a(n)e(f(n))=\mathcal{S}_1+\mathcal{S}_2-\mathcal{S}_3-\mathcal{S}_3,
$$
where 
\begin{align*}
\mathcal{S}_1&=O\left(\sum_{n\leq U}a(n)\right),\\
\mathcal{S}_2&=\sum_{l\leq V}b(l)\sum_{m\leq \frac{N}{l}}c(m)e(f(lm)),\\
\mathcal{S}_3&=\sum_{l\leq UV}\left(\sum_{\substack{rs=l\\r\leq V\\ s\leq U}}b(r)a(s)\right)\sum_{m\leq \frac{N}{l}}b^{-1}(m)e(f(lm)),\\
\mathcal{S}_4&=\sum_{V<l\leq \frac{N}{U}}\left(\sum_{\substack{r|l\\r\leq V}}b(r)b^{-1}\left(\frac{l}{r}\right)\right)\sum_{U<m\leq \frac{N}{l}}a(m)e(f(lm)).
\end{align*}

\end{proposition}
\begin{proof}
 The proof follows directly from the Lemma \ref{Vaughanlike}
\end{proof}
\section{ Proof of Theorem \ref{MT1} and Theorem \ref{MT2} }
\subsection{Proof of Theorem \ref{MT1}}
It is easy to see that $\tau(n)=1*1(n)$. Then in light of Proposition \ref{weighted}, we have 
\begin{align}\label{main-sum1}
\sum_{n\leq N} \tau(n)e(f(n))= \mathcal{T}_1 + \mathcal{T}_2 - \mathcal{T}_3 - \mathcal{T}_4,
\end{align}
where
\begin{align*}
    \mathcal{T}_1 =& O\left(\sum_{n\leq U}\tau(n)\right),\\
    \mathcal{T}_2 =& \sum_{l\leq V}\sum_{m\leq \frac{N}{l}}e(f(lm)),\\
    \mathcal{T}_3 =& \sum_{l<UV} \phi_1(l)\sum_{m\leq \frac{N}{l}}\mu(m)e(f(mn)),\\
    \mathcal{T}_4 =& \sum_{V<l\leq \frac{N}{U}}\phi_2(l)\sum_{U<m\leq \frac{N}{l}}\tau(m)e(f(lm));
\end{align*}
and 
\begin{align*}
    \phi_1(l) = \sum_{\substack{rs=t\\r< V\\s\leq U}}\tau(s)~\text{and}~\phi_2(l) = \sum_{\substack{r|l\\r\leq V}}\mu\left(\frac{l}{r}\right).
\end{align*}
We have 
\begin{align*}
              \phi_1(l)=(1*1*1)(l)\ll_{\epsilon}l^{\epsilon}.
\end{align*}
and
\begin{align*}
    \phi_2(l) = \sum_{\substack{r|l\\r\leq V}}\mu\left(\frac{l}{r}\right)\ll \sum_{\substack{r|l}}1\ll \tau(l)\ll_{\epsilon} l^{\epsilon}.
\end{align*}
At this point our goal is to estimate all $\mathcal{T}_i$'s. First let us fix the size of $U$ and $V$. For a fix $\theta$ take $U = V = N^{\theta}$.  For all four sums except the first one, we partition the intervals into dydic intervals. After dyadic subdivision, for each 
$\mathcal{T}_i$'s we have the upper bound of the form
$
\log N \max_{\substack{L\in [A, B]}} \Delta(L),
$
where
\begin{align}\label{delta}
\Delta(L) = \sum_{m \leq \frac{N}{L}} A(m) \sum_{\substack{l\sim L\\ml\leq N}} B(l)e(f(lm)).
\end{align}

For the case of $\mathcal{T}_2$, $A = 1, ~B = N^{\theta}, ~A(m) = 1$ and $B(l) = 1$. Here we use a work of Harman. We apply \cite[Corollary of Lemma 3]{HAR} if 
$$
L^R \geq \min \left\{q^{-1}N^{k}, ~q, ~N^{1-\theta}\right\},
$$
otherwise we use Lemma 4. Using Lemma 3 we get 
\begin{align*}
\Delta(L) \ll& N^{1+\epsilon}\bigg(\frac{1}{L^R} + \frac{1}{q} + \frac{q}{N^k} + \frac{1}{N^{1-\theta}}\bigg)^{\gamma}\\
          \ll& N^{1+\epsilon}\bigg(\frac{1}{q} + \frac{q}{N^k} + \frac{1}{N^{1-\theta}}\bigg)^{\gamma},
\end{align*}
where $\gamma=4^{1-k}$. If we use Lemma 4 then we get 
\begin{align*}
    \Delta(L) \ll& N^{1+\epsilon} L^{\frac{k-1}{R}}\bigg(\frac{1}{q}+\frac{q}{N^k}+\frac{1}{N^{1-\theta}}\bigg)^{1/R}\\
    \ll& N^{1+\epsilon} L^{\frac{k-1}{R}}\bigg(\frac{1}{q}+\frac{q}{N^k}+\frac{1}{N^{1-\theta}}\bigg)^{\frac{1}{R}-\gamma}\bigg(\frac{1}{q}+\frac{q}{N^k}+\frac{1}{N^{1-\theta}}\bigg)^{\gamma}\\
    \ll& N^{1+\epsilon} L^{\frac{k-1}{R}} \bigg(\frac{1}{L^R}\bigg)^{\frac{1}{R}-\gamma}\bigg(\frac{1}{q}+\frac{q}{N^k}+\frac{1}{N^{1-\theta}}\bigg)^{\gamma}\\
    \ll& N^{1+\epsilon}\bigg(\frac{1}{q}+\frac{q}{N^k}+\frac{1}{N^{1-\theta}}\bigg)^{\gamma} L^{\frac{k-1}{R}-1+\gamma R}.
\end{align*}
As $1-\frac{k-1}{R}\geq \gamma R$, we have $L^{\frac{k-1}{R}-1+\gamma R} \leq 1$. This implies Lemma $4$ finally gives
$$
\Delta(N) \ll N^{1+\epsilon}\bigg(\frac{1}{q}+\frac{q}{N^k}+\frac{1}{N^{1-\theta}}\bigg)^{\gamma}.
$$
Combining both we get the estimate for $\mathcal{T}_2$,
\begin{equation}\label{sumtwo}
    \mathcal{T}_2 \ll \log N\cdot N^{1+\epsilon} \bigg(\frac{1}{q}+\frac{q}{N^k}+\frac{1}{N^{1-\theta}}\bigg)^{\gamma}.
\end{equation}
Following the same method with some minor variation one can show that 
\begin{equation}\label{sumthreefour}
    \mathcal{T}_i\ll \log N\cdot N^{1+\frac{4\epsilon}{3}} \bigg(\frac{1}{q}+\frac{q}{N^k}+\frac{1}{N^{1-\theta}}\bigg)^{\gamma},
\end{equation}
for $i=3,4$. Now if we start with $\epsilon/{100}>0$, then putting bounds \eqref{sumtwo} and \eqref{sumthreefour} for $\mathcal{T}_i$'s in \eqref{main-sum1}, we finally get
\begin{equation*}\label{final1}
   \sum_{n\leq N} \tau(n)e(f(n)) \ll \log N \cdot N^{1+\epsilon} \bigg(\frac{1}{q}+\frac{q}{N^k}+\frac{1}{N^{1-\theta}}\bigg)^{\gamma}.
\end{equation*}
\subsection{Proof of Theorem \ref{MT2}}
We now calculate $\sum_{n\leq N} \mu^2(n)e(f(n))$. Let us first observe that 
$$
\mu^2(n)=1*\nu(n),
$$
where 
\[
    \nu(n)= 
\begin{cases}
    \mu(\sqrt{n}),& \text{if } n \quad \text{is a square};\\
    0,              & \text{otherwise}.
\end{cases}
\]
Hence as before we can write this summation as follows
$$
\sum_{n\leq N} \mu^2(n)e(f(n)) = \mathcal{M}_1 + \mathcal{M}_2 - \mathcal{M}_3 - \mathcal{M}_4,
$$
where
\begin{align*}
    \mathcal{M}_1 =& O\left(\sum_{n\leq U}\mu^2(n)\right),\\
    \mathcal{M}_2 =& \sum_{l\leq U}\sum_{m\leq \frac{N}{l}}\nu(m)e(f(lm)),\\
    \mathcal{M}_3 =& \sum_{l<U^2} \phi_3(l)\sum_{m\leq \frac{N}{l}}\mu(m)e(f(mn)),\\
    \mathcal{M}_4 =& \sum_{U<l\leq \frac{N}{U}}\phi_4(l)\sum_{U<m\leq \frac{N}{m}}\mu^2(m)e(f(lm)),
\end{align*}
and
\begin{align*}
    \phi_3(l) = \sum_{\substack{rs=l\\r< U\\s\leq U}}\nu(s)
              = \ll\tau(l)
               \ll_{\epsilon}l^{\epsilon}.
\end{align*}
and
\begin{align*}
    \phi_2(l) = \sum_{\substack{r|l\\r\leq V}}\mu\left(\frac{l}{r}\right)
               \ll \sum_{\substack{r|l}}1
               \ll \tau(l)\ll_{\epsilon} l^{\epsilon}.
\end{align*}
Now as before we may start with $\epsilon/{100}>0$ instead of $\epsilon$ to show in the similar fashion as in the proof of Theorem \ref{MT1},
\begin{equation*}\label{final2}
   \sum_{n\leq N} \mu^2(n)e(f(n)) \ll  N^{1+\epsilon} \bigg(\frac{1}{q}+\frac{q}{N^k}+\frac{1}{N^{1-\theta}}\bigg)^{\gamma}
\end{equation*}
which completes the proof of Theorem \ref{MT2}.
\section{Proof of Theorem \ref{MT3} and \ref{MT4}}
Let $$S_{h, \tau}=\sum_{N^{1/2}\leq n\leq  N}\tau(n)e(h f(n)).$$
We want to use Lemma \ref{lem-h3}, i.e., our aim is to establish
\begin{align*}
\sum_{h=1}^{H}\left|\sum_{\substack{N^{1/2}<n\leq N\\\tau(n)>2}}\tau(n)e(hf(n))\right|<\frac{1}{6}\sum_{\substack{N^{1/2}<n\leq N\\\tau(n)>2}}\tau(n).
\end{align*}
Then it will imply that $\exists~ n\in [N^{1/2},N]\cap \{n\in\mathbb{N}:\tau(n)>2\}$ such that 
\begin{align*}
||f(n)||\leq \frac{1}{H}.
\end{align*} 
Now we have
\begin{align*}
\sum_{\substack{N^{1/2}<n\leq N\\\tau(n)>2}}\tau(n)=\sum_{\substack{N^{1/2}<n\leq N}}\tau(n)-\sum_{\substack{N^{1/2}<n\leq N\\\tau(n)\leq 2}}\tau(n).
\end{align*}
The last sum in the right hand side is essentially running over primes. Hence 
\begin{align*}
\sum_{\substack{N^{1/2}<n\leq N\\\tau(n)>2}}\tau(n) \geq  \sum_{\substack{N^{1/2}<n\leq N}}\tau(n)-2\sum_{\substack{N^{1/2}<p\leq N}}1 \gg N\log N,
\end{align*}
Now using triangle inequality we have 
\begin{align}\label{total}
\sum_{h=1}^{H}\left|\sum_{\substack{N^{1/2}<n\leq N\\\tau(n)>2}}\tau(n)e(hf(n))\right|&\leq  \sum_{h=1}^{H}\left|\sum_{\substack{N^{1/2}<n\leq N}}\tau(n)e(hf(n))\right|\notag\\
&\hspace{1cm}+\sum_{h=1}^{H}\left|\sum_{\substack{N^{1/2}<n\leq N\\\tau(n)\leq 2}}\tau(n)e(hf(n))\right|.
\end{align}
Again the inner summation of the  second part runs over the primes. Hence
\begin{align*}
\sum_{h=1}^{H}\left|\sum_{\substack{N^{1/2}<n\leq N\\\tau(n)\leq 2}}\tau(n)e(hf(n))\right|\leq 2\sum_{h=1}^{H}\left|\sum_{\substack{N^{1/2}<p\leq N}}e(hf(p))\right|.
\end{align*}
By using partial summation formula we have 
\begin{align*}
&\sum_{h=1}^{H}\left|\sum_{\substack{N^{1/2}<n\leq N}}e(hf(n))\right|\\&=\sum_{h=1}^{H}\left|\frac{1}{\log N}\sum_{p\leq N}\log p\cdot e\left(hf(p)\right)-\frac{2}{\log N}\sum_{p\leq N^{1/2}}\log p\cdot e\left(hf(p)\right)\right.\\&\left.\hspace{7cm}+\int_{N^{1/2}}^{N}\sum_{p\leq t}\log p\cdot e\left(hf(p)\right)\frac{dt}{t\log^2 t}\right|\\
&\leq \frac{1}{\log N}\sum_{h\leq H}\left|\sum_{N^{1/2}<p\leq N}\log p\cdot e\left(hf(p)\right)\right|\\&\hspace{3.5cm}+\frac{1}{\log N}\max_{N^{1/2}<t< N}\sum_{h\leq H}\left|\sum_{t^{1/2}<p<t}\log p\cdot e(hf(p))\right|+2\frac{N^{\frac{1}{2}}H}{\log N}.
\end{align*}
To this end we want to find upper bounds for the following sums
\begin{align}\label{eq1}
\sum_{h\leq H}\left|\sum_{N^{1/2}<n<N}\tau(n)e\left(hf(n)\right)\right|,
\end{align}
\begin{align}\label{eq2}
\sum_{h\leq H}\left|\sum_{N^{1/2}<p<N}\log p\cdot e\left(hf(p)\right)\right|,
\end{align}
\begin{align}\label{eq3}
\max_{N^{1/2}<t\leq N}\sum_{h\leq H}\left|\sum_{t^{1/2}<p<t}\log p\cdot e\left(hf(p)\right)\right|.
\end{align}
On the other hand,
let $$S_{h,\mu^2}=\sum_{N^{1/2}\leq n\leq  N}\mu^2(n)e(h f(n)).$$
Again, we want to use Lemma \ref{lem-h3}, i.e., our aim is to establish
\begin{align*}
\sum_{h=1}^{H}\left|\sum_{n\in \mathcal{A}}\mu^2(n)e(hf(n))\right|<\frac{1}{6}\sum_{n \in \mathcal{A}}\mu^2(n),
\end{align*}
Where 
$$
\mathcal{A}= \{N^{1/2}<n\leq N: (n, P(N^{1/2}))>1\}.
$$
Then by Lemma  \ref{lem-h3} it will imply that $\exists~ n\in \mathcal{A}$ such that 
\begin{align*}
||f(n)||\leq \frac{1}{H}.
\end{align*} 
Now we have
\begin{align*}
\sum_{\substack{N^{1/2}<n\leq N\\ n\in \mathcal{A}}}\mu^2(n)=\sum_{\substack{N^{1/2}<n\leq N}}\mu^2(n)-\sum_{\substack{N^{1/2}<n\leq N\\ n\in \mathcal{B}}}\mu^2(n)
\end{align*}
where
$$
\mathcal{B}= \{N^{1/2}<n\leq N: (n,P(N^{1/2}))=1\}.
$$
Now the set $\mathcal{B}$ precisely consists of all primes in between $N^{1/2}$ and $N$. Hence
\begin{align*}
\sum_{\substack{N^{1/2}<n\leq N\\n \in \mathcal{A}}}\mu^2(n) \geq  \sum_{\substack{N^{1/2}<n\leq N}}\mu^2(n)-\sum_{\substack{N^{1/2}<p\leq N}}1\gg  N,
\end{align*}
Now as before using triangle inequality to have an upper bound for 
\begin{align*}
\sum_{h=1}^{H}\left|\sum_{n\in \mathcal{A}}\mu^2(n)e(hf(n))\right|
\end{align*}
it amounts to find out upper bounds for the following quantities 
\begin{align}\label{eq4}
\sum_{h\leq H}\left|\sum_{N^{1/2}<n<N}\mu^2(n)e\left(hf(n)\right)\right|,
\end{align}
\eqref{eq2} and \eqref{eq3}. \par
Our main aim is to estimate $\sum_{h=1}^{H}|S_{h, \tau}|$ and $\sum_{h=1}^{H}|S_{h, \mu^2}|$. We will describe the method for one, say $\sum_{h=1}^{H}|S_{h, \tau}|$ . The other will follow similarly. We noticed in Section 3 that $|S_{h, \tau}|$ can further split into four sums and except the first one all can be further divided into sub-sums over dyadic intervals. First we take $$U = V = N^{1/2}.$$ As \eqref{delta}, we need to now take care of sums of like 
$$
\sum_{h=1}^{H}\left|\sum_{m \leq \frac{N}{L}} A(m) \sum_{\substack{l\sim L\\ml\leq N}} B(l)e(hf(lm))\right|.
$$
Let us denote innermost sum by $\Delta(L, h)$ . So if $R'$ is such that $\frac{1}{R'}+\frac{1}{R^2}=1$, then applying H\"{o}lder with the weights $R'$ and $R^2$ and using Lemma \ref{lem-h1} $(W=\frac{N}{L}, X=L)$, we get
\begin{align*}
\sum_{h=1}^{H}|\Delta(L, h)|&\leq \left(\sum_{h=1}^{H}1\right)^{1/R'}\cdot\left(\sum_{h=1}^{H}|\Delta(L, h)|^{R^2}\right)^{1/R^2}\\
&\ll H^{1/R'}\frac{N}{L}L\left(\sum_{h=1}^{H}L^{-R}+(N)^{-k+\epsilon/2}\sum_{h=1}^{H}\sum_{z=1}^{Y}\min\left(\frac{N}{L},\frac{1}{||\alpha h z||}\right)\right)^{1/R^2}\\
&\ll H^{1/R'}N\left(\sum_{h=1}^{H}L^{-R}+(N)^{-k+\epsilon/2}\sum_{r=1}^{HY}\left(\sum_{hz=r}1\right)\min\left(\frac{N}{L},\frac{1}{||\alpha r||}\right)\right)^{1/R^2}.
\end{align*}
Hence putting $Y=L^k(\frac{N}{L})^{k-1}(k!)^2$ and putting the bound for divisor function one can derive
\begin{align}\label{eq-a1}
\sum_{h=1}^{H}|\Delta(L, h)|\ll (NH)^{1+\epsilon/2}{\left(L^{-R}+\frac{1}{q}+\frac{1}{N^{\frac{1}{2}}}+\frac{q}{HN^{k}}\right)}^{\gamma}.
\end{align}
Now using H\"{o}lder inequality again we get
\begin{align*}
\sum_{h=1}^{H}|\Delta(L, h)|&\leq \left(\sum_{h=1}^{H}1\right)^{1/R''}\cdot\left(\sum_{h=1}^{H}|\Delta(L, h)|^{R}\right)^{1/R},
\end{align*}
where $\frac{1}{R''}+\frac{1}{R}=1$.
Now using Lemma \ref{lem-h2} we get
\begin{align*}
\sum_{h=1}^{H}|\Delta(L, h)|\ll H^{1/R^{''}}\left(\sum_{h=1}^{H}L^{\epsilon/2+R-1}\sum_{v=1}^{L}\sum_{y=1}^{k!(\frac{N}{L})^{k-1}}(\frac{N}{L})^{R-k+\epsilon/2}\min\left(\frac{N}{L},\frac{1}{||\alpha h y v^k||}\right)\right)^{1/R}.
\end{align*}
Now clubbing  $h, y, v^k$ together and using the upper bound of $\tau_3$  one can get 
\begin{align}\label{eq-a2}
\sum_{h=1}^{H}|\Delta(L, h)|\ll (NH)^{1+\epsilon/2}X^{\frac{k-1}{R}}\left(\frac{1}{q}+\frac{1}{N^{\frac{1}{2}}}+\frac{q}{N^kH}\right)^{1/R}.
\end{align}

We use the bound in \eqref{eq-a1} for 
\begin{align*}
X > \min \{N^{(1-\theta)/R}, q^{1/r}, NL^{1/k}q^{-1/k}\}.
\end{align*}
Otherwise we use \eqref{eq-a2}. In either case we find 
$$
\sum_{h=1}^{H}|\Delta(L, h)| \ll (NH)^{1+\epsilon/2}{\left(\frac{1}{q}+\frac{1}{N^{\frac{1}{2}}}+\frac{q}{HN^{k}}\right)}^{\gamma}.
$$
Now combining all four parts we get 
\begin{align*}
    \sum_{h=1}^{H}\left|\sum_{\substack{N^{1/2}<n\leq N\\\tau(n)>2}}\tau(n)e(hf(n))\right| \ll& \log N \cdot N^{1/2}H +\\& \log N \cdot {(NH)}^{1+\epsilon/2}{\big(q^{-1/2}+N^{-1/2}+qN^{-k}H^{-1}\big)}^{\gamma},
\end{align*}
and similarly we get 
\begin{align*}
    \sum_{h=1}^{H}\left|\sum_{\substack{N^{1/2}<n\leq N\\n\in \mathcal{A}}}\mu^2(n)e(hf(n))\right| \ll& N^{1/2}H +\\& \log N \cdot {(NH)}^{1+\epsilon/2}{\big(q^{-1/2}+N^{-1/2}+qN^{-k}H^{-1}\big)}^{\gamma}.
\end{align*}
Using \cite{HAR}, 
we get an upper bound for \eqref{eq2} as
\begin{align*}
    \sum_{h\leq H}\left|\sum_{N^{1/2}<p<N}\log p\cdot e\left(hf(p)\right)\right| \ll & N^{1/2}H + \\& \log N \cdot  {(NH)}^{1+\epsilon/2}{\big(q^{-1/2}+N^{-1/2}+qN^{-k}H^{-1}\big)}^{\gamma}.
\end{align*}
With a slight variation, but essentially by the same method one can also get an upper bound for \eqref{eq3} as
\begin{align*}
    &\max_{N^{1/2}<t\leq N}\sum_{h\leq H}\left|\sum_{t^{1/2}<p<t}\log p\cdot e\left(hf(p)\right)\right| \ll N^{1/2}H \\&\hspace{3cm}+  \log N \cdot  {(NH)}^{1+\epsilon/2}{\big(q^{-1/2}+N^{-1/4}+qN^{-k/2}H^{-1}\big)}^{\gamma}.
\end{align*}
Now we choose $q=N^{1/2}$ and $H=N^{\gamma/4-\epsilon}$. Hence combining all upper bounds of \eqref{eq1}, \eqref{eq2}, \eqref{eq3} and \eqref{eq4}, we have 
\begin{align*}
    \sum_{h=1}^{H}\left|\sum_{\substack{N^{1/2}<n\leq N\\\tau(n)>2}}\tau(n)e(hf(n))\right| \ll \log N \cdot N^{1-\epsilon^2/2}\ll \frac{1}{6} N\cdot \log N < \frac{1}{6} \sum_{\substack{N^{1/2}<n\leq N\\\tau(n)>2}}\tau(n)
\end{align*}
and 
\begin{align*}
    \sum_{h=1}^{H}\left|\sum_{\substack{N^{1/2}<n\leq N\\n\in \mathcal{A}}}\mu^2(n)e(hf(n))\right| &\ll N^{1-\epsilon^2/4}\ll \frac{1}{6} N < \frac{1}{6} \sum_{\substack{N^{1/2}<n\leq N\\n\in \mathcal{A}}}\mu^2(n),
\end{align*}
which using Lemma \ref{lem-h3}  establish Theorem \ref{MT3} and \ref{MT4}.

\section{Acknowledgement}
The first author would like to thank Department of Mathematics, Thapar Institute of Engineering and Technology and the second author would like to thank Chennai Mathematical Institute for providing excellent working conditions.  During the preparation of this article, D.M. was supported
by the National Board of Higher Mathematics post-doctoral fellowship (No.: 0204/10/(8)/2023/R\&D-II/2778).


\end{document}